\documentclass[11pt]{amsart}
\usepackage[top=1.25in,bottom=1.25in,left=1in,right=1in]{geometry}
\usepackage{amsfonts, amsmath, amssymb, amscd, amsthm, graphicx}
\usepackage[final, 
colorlinks=true,
pdftitle={Hyperbolic quotients of projection complexes},
pdfauthor={Matt Clay, Johanna Mangahas}]{hyperref}
\usepackage{enumitem,diagbox,multirow}

\usepackage{tikz}
\usetikzlibrary{calc}
\makeatletter
\newcommand{\gettikzxy}[3]{%
  \tikz@s\can@one@point\pgfutil@firstofone#1\relax
  \edef#2{\the\pgf@x}%
  \edef#3{\the\pgf@y}%
}
\makeatother

\newtheorem{theorem}{Theorem}
\newtheorem{lemma}[theorem]{Lemma}
\newtheorem{proposition}[theorem]{Proposition}
\newtheorem{corollary}[theorem]{Corollary}

\theoremstyle{definition}

\numberwithin{theorem}{section}

\newcommand{\Mod}{\mathrm{Mod}}

\newcommand{\C}{\mathcal{C}}

\newcommand{\calO}{\mathcal{O}}
\renewcommand{\P}{\mathcal{P}}

\newcommand{\grp}[1]{\langle{#1}\rangle}
\newcommand{\nc}[1]{\left\langle\hspace{-.7mm}\left\langle{#1}\right\rangle\hspace{-.7mm}\right\rangle}
\newcommand{\N}{\mathbb{N}}
\newcommand{\Z}{\mathbb{Z}}
\newcommand{\R}{\mathbb{R}}
\newcommand{\Stab}{\mathrm{Stab}}

\newcommand{\Ke}{C_{\rm e}}
\newcommand{\Kp}{C_{\rm p}}
\newcommand{\Kg}{C_{\rm g}}
\newcommand{\Kwpd}{C_{\rm WPD}}

\newcommand{\Lwind}{L_0}
\newcommand{\Lhyp}{L_{\rm hyp}}
\newcommand{\Llift}{L_{\rm lift}}
\newcommand{\Lwpd}{L_{\rm WPD}}
\newcommand{\Lshort}{L_{\rm short}}
\newcommand{\Lpro}{L_{\rm pro}}

\newcommand{\Y}{\mathbb{Y}}
\newcommand{\from}{\colon\thinspace}

\newcommand{\BigFreeProd}[1]{\raisebox{-1.5pt}{ \ensuremath{\underset{\mbox{\scriptsize{$#1$}}}{\mbox{\Huge{$\ast$}}}}}\,}

\newcommand{\p}[1]{\medskip\paragraph{{\it #1}}}
\newcommand{\ip}[1]{\medskip\paragraph{{\em #1}}}

\newcommand{\hh}{\text{\sf h}}
\newcommand{\ii}{\text{\sf i}}
\newcommand{\nn}{\text{\sf n}}

\DeclareMathOperator{\EC}{EC}
\DeclareMathOperator{\Piv}{Piv}

\newcommand{\Pive}{\Piv^{*}}

\title[Hyperbolic quotients of projection complexes]{Hyperbolic quotients of projection complexes}
\author{Matt Clay}
\author{Johanna Mangahas}
\date{} 

\begin{document}

\begin{abstract}
This paper is a continuation of our previous work with Margalit where we studied group actions on projection complexes.  In that paper, we demonstrated sufficient conditions so that the normal closure of a family of subgroups of vertex stabilizers is a free product of certain conjugates of these subgroups.  In this paper, we study both the quotient of the projection complex by this normal subgroup and the action of the quotient group on the quotient of the projection complex.  We show that under certain conditions that the quotient complex is $\delta$--hyperbolic.  Additionally, under certain circumstances, we show that if the original action on the projection complex was a non-elementary WPD action, then so is the action of the quotient group on the quotient of the projection complex.  This implies that the quotient group is acylindrically hyperbolic.
\end{abstract}

\maketitle


\section{Introduction}\label{sec:intro}

Projection complexes were originally defined by Bestvina--Bromberg--Fujiwara and were used to show that the mapping class group of an orientable surface has finite asymptotic dimension~\cite{ar:BBF15}.  The motivating idea behind these complexes is the following.  Start with a collection of subspaces $\{Z_i\}$ contained in some metric space $X$.  We want these subspaces to satisfy some properties akin to negative curvature; in particular, we require that the nearest point projection from any one subset $Z_i$ to another subset $Z_j$ has uniformly bounded diameter.  For example, one could take $X$ to be the hyperbolic plane and the collection $\{Z_i\}$ to be the orbit of a geodesic in $X$ under a discrete group of isometries of $X$.  The projection complex built out of this data is the graph with vertex set $\{Z_i\}$ where two vertices $Z_i$ and $Z_j$ are joined by an edge if the diameter of the union of their projections to any other $Z_k$ is small.  A key feature of a projection complex is that, in general, it is a quasi-tree, in other words, it is quasi-isometric to a tree~\cite[Theorem~3.16]{ar:BBF15}.  Projection complexes have found several useful applications lately by many authors~\cite{ar:DGO17,ar:Dahmani18,ar:BBFS20,un:BBF,un:DHS,un:GH,un:HQR}.      

In previous work with Margalit, we studied group actions on projection complexes~\cite{un:CMM}.  We derived a structure theorem for normal subgroups generated by elliptic elements under some hypotheses; see Section~\ref{sec:windmills} for the exact statement.  We were able to apply our structure theorem to produce new examples of normal subgroups of mapping class groups of orientable surfaces that are isomorphic to right-angled Artin groups.  In particular, we produced examples that were not free.  

In this paper, we work in the general setting of a group acting on a projection complex with the same set of hypotheses as before and study both the quotients of the projection complex by such normal subgroups and the action of the quotient group on the corresponding quotient complex.  These appear as Theorem~\ref{thm:hyperbolic quotient} and Theorem~\ref{thm:wpd action on quotient} respectively.  To state these results, we describe the set up we studied before and continue to study in this paper.    

Briefly, a projection complex is a graph $\P$ and a collection of functions
\[
d_v \from V \setminus \{v\} \times V \setminus \{v\} \to \R_{\geq 0}
\]
where $V$ is the set of vertices of $\P$ and $v \in V$.   The full definition appears in Section~\ref{sec:projection complex}.  Following our previous work and as explained below, our definition is a mild modification of the original definition of Bestvina--Bromberg--Fujiwara.

Let $\P$ be a projection complex, and let $G$ be a group that acts on $\P$.    Further, for each vertex $v$ of $\P$, let $R_v$ be a subgroup of the stabilizer of $v$ in $G$.  Let $L > 0$.  We say that the family of subgroups $\{R_v\}$ is an \emph{equivariant $L$--spinning family} of subgroups of $G$ if it satisfies the following two conditions:

\begin{itemize}\itemindent=-22pt
\item \emph{Equivariance:} If $g$ lies in $G$ and $v$ is a vertex of $\P$ then
\[
gR_vg^{-1} = R_{gv}.
\]  
\item \emph{Spinning:} For any distinct vertices $v$ and $w$ of $\P$ and any nontrivial $h \in R_v$ we have
\[
d_v(w,hw) \geq L.
\]
\end{itemize}
By the equivariance condition, for each vertex $v$ the subgroup $R_v$ is normal in $\Stab_G(v)$, and the subgroup $H$ of $G$ generated by the $R_v$ is normal in $G$.  If $\{v_i\}$ is a set of orbit representatives for the action of $G$ on the vertices of $\P$, then $H$ is the normal closure of the set $\left\{R_{v_i}\right\}$. 

We can now state our theorem regarding the quotient complex.

\begin{theorem}\label{thm:hyperbolic quotient}
Let $\P$ be a projection complex and let $G$ be a group acting on $\P$.  There exists a constant $\Lhyp(\P)$ with the following property.  If $L \geq \Lhyp(\P)$ and  if $\{R_v\}$ is an equivariant $L$--spinning family of subgroups of $G$ then $\P/\grp{R_v}$ is $\delta$--hyperbolic.
\end{theorem}

We also examine the action of the quotient group $G/\grp{R_v}$ on the quotient space $\P/\grp{R_v}$.  Our result on the action briefly says that certain features of the action of $G$ on $\P$ persist in the quotient action.  Before we can state our result on this action, we need to define a number of notions.  

Let $X$ be a geodesic metric space and let $G$ be a group that acts on $X$ by isometries.  An element $f$ of $G$ is \emph{hyperbolic} if 
\[ \lim_{n \to \infty} \frac{d(x,f^nx)}{n} \] is positive for some $x \in X$, equivalently, for any $x \in X$.  Two hyperbolic elements, $f_1$ and $f_2$, of $G$ are \emph{independent} if $d(f_1^{n_1}x,f_2^{n_2}x) \to \infty$ as $n_1,n_2 \to \pm \infty$ for some $x \in X$, equivalently, for any $x \in X$.  

An element $f$ of $G$ is a \emph{WPD element} if $f$ is hyperbolic and if for all points $x \in X$ and for all $D \geq 0$, there is an $M \geq 0$ such that the set 
\[ \{g \in G \mid d(x,gx) \leq D \mbox{ and } d(f^M x, gf^Mx ) \leq D \} \] is finite.  We remark that it suffices to demonstrate finiteness of the above set at a single point in $X$.  The notion of a WPD element was introduced by Bestvina--Fujiwara as a tool for constructing quasi-morphisms~\cite{ar:BF02}.  There are several known examples of WPD elements: pseudo-Anosov mapping classes acting on the corresponding curve complex~\cite{ar:BF02} and fully irreducible outer automorphisms of a free group acting on the corresponding free factor complex~\cite{ar:BF14-1} for instance.  If the action of $G$ on $X$ is properly discontinuous, then any hyperbolic element is a WPD element.        

The action of $G$ on $X$ is a \emph{non-elementary WPD action} if there exist two elements in $G$ that are WPD elements and independent.  We remark that if $X$ is $\delta$--hyperbolic, $G$ is not virtually cyclic and there is one element $f$ in $G$ that is a WPD element, then for some element $g$ in $G$, the elements $f$ and $gfg^{-1}$ are independent WPD elements.  In fact, one can take any $g \in G$ such that $\grp{f} \cap g\grp{f}g^{-1}$ is finite. 

We can now state our theorem on the action of the quotient group on the quotient complex. 

\begin{theorem}\label{thm:wpd action on quotient}
Let $\P$ be a projection complex and let $G$ be a group with a non-elementary WPD action on $\P$.  There exists a constant $\Lwpd(\P,G)$ with the following property.  If $L \geq \Lwpd(\P,G)$ and if $\{R_v\}$ is an equivariant $L$--spinning family of subgroups of $G$ then the action of  $G/\grp{R_v}$ on $\P/\grp{R_v}$ is a non-elementary WPD action.  

Precisely, if $f_1$ and $f_2$ are independent WPD elements of $G$ for its action on $\P$, then there is a constant $\Lwpd(\P,f_1,f_2)$ such that their images $\bar f_1$ and $\bar f_2$ in $G/\grp{R_v}$ are independent WPD elements for the action of $G/\grp{R_v}$ on $\P/\grp{R_v}$ when $L \geq \Lwpd(\P,f_1,f_2)$ and $\{R_v\}$ is an equivariant $L$--spinning family of subgroups of $G$.
\end{theorem}

Whereas the constant in Theorem~\ref{thm:hyperbolic quotient} does not depend on $G$, the constant in Theorem~\ref{thm:wpd action on quotient} necessarily does.  Indeed, if $G$ is equal to $\grp{R_v}$ then the quotient group is trivial.  Hence we must choose $L$ after $G$---more precisely after choosing two independent WPD elements---to ensure that the quotient is as claimed.

There is a strengthening of the WPD condition called acylindricity that arises in several settings that we describe now.

Let $X$ be a metric space and let $G$ be a group acting on $X$ by isometries.  The action is \emph{acylindrical} if for all $D \geq 0$ there exist $R \geq 0$ and $N \geq 0$ such that for all points $x$ and $y$ in $X$ where $d(x,y) \geq R$, the set
\[ \{ g \in G \mid d(x,gx) \leq D \mbox{ and } d(y,gy) \leq D \}  \] contains at most $N$ elements.

A group $G$ is \emph{acylindrically hyperbolic} if it admits an acylindrical action on a hyperbolic space for which there exist elements $f_1$ and $f_2$ in $G$ that are hyperbolic and independent.  Both the mapping class group of an orientable surface~\cite{ar:Bowditch08} and the outer automorphism group of a finitely generated free group are acylindrically hyperbolic~\cite{ar:Osin16}.  There are several other examples and much is known about this class of groups.  The paper by Osin contains a survey of examples and results for acylindrically hyperbolic groups~\cite{ar:Osin16}.

Osin derived a number of conditions that are equivalent to acylindrical hyperbolicity, one of which is that the group is not virtually cyclic and admits an action on a $\delta$--hyperbolic space where one element is a WPD element~\cite[Theorem~1.2]{ar:Osin16}.  Hence we obtain the following corollary of Theorems~\ref{thm:hyperbolic quotient} and ~\ref{thm:wpd action on quotient}.

\begin{corollary}\label{co:acylindrically}
Let $\P$ be a projection complex and let $G$ be a group with non-elementary WPD action on $\P$.  There exists a constant $\Lwpd(\P,G)$ with the following property.  If $L \geq \Lwpd(\P,G)$ and if $\{R_v\}$ is an equivariant $L$--spinning family of subgroups of $G$ then $G/\grp{R_v}$ is acylindrically hyperbolic.
\end{corollary}

In Section~\ref{sec:examples} we describe new examples of acylindrically hyperbolic groups coming from this construction.  These groups are quotients of the mapping class group of an orientable surfaces by the normal subgroups we produced in our previous work.

The strategy to prove Theorem~\ref{thm:hyperbolic quotient} and Theorem~\ref{thm:wpd action on quotient} is very similar to strategy of Dahmani--Hagen--Sisto in a recent paper~\cite{un:DHS}.  In this paper, Dahmani--Hagen--Sisto consider the action of the subgroup of the mapping class group generated by $k$th powers of Dehn twists on the curve graph, i.e., the 1--skeleton of the curve complex and they prove results similar to Theorem~\ref{thm:hyperbolic quotient} and Theorem~\ref{thm:wpd action on quotient}.  They make use of the fact shown by Dahmani that the curve graph has the structure of a \emph{composite projection graph}~\cite{ar:Dahmani18}.  That is, there is a partition of the curve graph into finitely many pieces that behave like a projection complex, along with certain combatility conditions on how the pieces interact.  Since we deal with a projection complex as opposed to a composite projection graph, some parts of their strategy can be simplified.

In order to prove Theorem~\ref{thm:hyperbolic quotient}, we show that geodesic triangles in $\P/\grp{R_v}$ lift to geodesic triangles in $\P$ (Proposition~\ref{prop:lift quad}).  As $\P$ is a quasi-tree, it is a $\delta$--hyperbolic metric space and hence geodesic triangles are $\delta_0$--thin for some $\delta_0$.  As the quotient map $p \from \P \to \P/\grp{R_v}$ is 1--Lipschitz, this shows that the geodesic triangles in $\P/\grp{R_v}$ are $\delta_0$--thin as well.  The proof of Theorem~\ref{thm:wpd action on quotient} is similar except that it involves lifting geodesic quadrilaterals.  A key fact needed here is that a geodesic in $\P$ for which the projection of any two of its vertices to any other vertex in $\P$ is uniformly bounded is isometrically embedded in the quotient (Lemma~\ref{lem:bounded projections}). 

A closed path in $\P/\grp{R_v}$ can be lifted to a path in $\P$ with endpoints $x$ and $hx$ for some $h \in \grp{R_v}$.  We describe a technique called \emph{path bending} for replacing the lifted path with a new lift.  There is a notion of \emph{complexity} for an element in $\grp{R_v}$.  We show that when $x \neq hx$, we can bend the given lift to get a lift from $x$ to $h'x$ where $h'$ has less complexity than that of $h$ (Proposition~\ref{prop:shorten}).  This is the technique known as \emph{shortening} and it plays a key role in understanding both lifts (Proposition~\ref{prop:lift quad}) and images (Lemma~\ref{lem:project geo}).  This technique was introduced by Dahmani--Hagen--Sisto and is also essential to their work~\cite{un:DHS}. 

\subsection{Outline of Paper} Section~\ref{sec:prelims} collects the necessary facts on projection complexes that are needed for the remainder.  Starting in Section~\ref{sec:shortening}, we follow the strategy of Dahmani--Hagen--Sisto~\cite{un:DHS}. In Section~\ref{sec:shortening}, we prove the main technical tool of the paper, Proposition~\ref{prop:shorten}.  This is the technique known as shortening and allows us to replace a lift of a path in the quotient of the projection complex with another lift that is simpler in a precise sense.  We apply the shortening tool in Section~\ref{sec:lifting} to show that geodesic quadrilaterals in the quotient of the projection complex lift to geodesic quadrilaterals.  The proof of Theorem~\ref{thm:hyperbolic quotient} appears in Section~\ref{sec:proof hyperbolic}.  In Section~\ref{sec:small}, we show that when vertices along a geodesic in the projection complex have bounded projections, the image of the geodesic in the quotient graph is still a geodesic.  Using this, we can establish that certain WPD elements for the action of $G$ on $\P$ have images in $G/\grp{R_v}$ that are still WPD elements for the action of $G/\grp{R_v}$ on $\P/\grp{R_v}$.  In Section~\ref{sec:proof wpd action}, we prove Theorem~\ref{thm:wpd action on quotient}.  Finally, in Section~\ref{sec:examples} we present some examples when $G$ is the mapping class group of a surface.  

\subsection{Acknowledgments} We would like to thank Alessandro Sisto for suggesting that the techniques of his paper with Dahmani and Hagen could apply in our setting as well.  We are immensely grateful to Dan Margalit for initiating our projects on windmills in projection complexes and for ideas, questions, and conversations.  We thank the anonymous referee for reading our paper carefully and for providing useful comments.

The first author is partially supported by the Simons Foundation Grant No.~316383. The second author is supported by National Science Foundation Grant No.~DMS--1812021.  

\section{Projection complexes, Windmills and Pivot Points}\label{sec:prelims}

In this section we provide the definitions of projection complexes, windmills and pivot points.  The majority of the discussion in this section appears in our previous work with Margalit~\cite{un:CMM}.  The essential material that is needed for the sequel is recorded in Lemma~\ref{lem:pivot point facts}.

\subsection{Projection Complexes}\label{sec:projection complex}

We begin with the definition of a projection complex.  Let $\Y$ be a set and let $\theta \geq 0$ be a constant.   Assume that for each $y \in \Y$ there is a function
\[
d_{y} \from \Y \setminus \{y\} \times \Y \setminus \{y\} \to \R_{\geq 0}
\]
with the following properties.

\medskip

\begin{enumerate}[leftmargin=0.5cm,itemsep=.5em]
\item[] {\it Symmetry:} \ $d_{y}(x,z) = d_{y}(z,x)$ for all $x,y,z \in \Y$
\item[] {\it Triangle inequality:} \ $d_{y}(x,z) + d_{y}(z,w) \geq d_{y}(x,w)$ for all $x,y,z,w \in \Y$
\item[] {\it Inequality on triples:} \ $\min \{ d_{y}(x,z), d_{z}(x,y)  \} \leq \theta$ for all $x,y,z \in \Y$
\item[] {\it Finiteness:} \ $\#\{ y \in \Y \mid d_{y}(x,z) > \theta  \}$ is finite  for all $x,z \in \Y$
\end{enumerate}

\medskip 

\noindent These conditions are known as the \emph{projection complex axioms}.  When we say that a set $\Y$ and a collection of functions $\{d_y\}_{y \in \Y}$ as above satisfy the projection complex axioms the constant $\theta$ is implicit.  

For a given $K \geq  0$, we will define a graph $\P_K(\Y)$ with vertices corresponding to the elements in $\Y$.  The edges are defined using the notion of modified distance functions.

Given the functions $\{d_y\}$, Bestvina--Bromberg--Fujiwara \cite{ar:BBF15} constructed another collection of functions $\{d_{y}'\}_{y \in \Y}$, where each $d_y'$ shares the same domain and target as $d_y$.  Because the definition of the $d_y'$ is technical and because we do not use the definition in this paper, we do not state it here.  Bestvina--Bromberg--Fujiwara~\cite[Theorem~3.3B]{ar:BBF15} showed that the modified functions are coarsely equivalent to the original functions: for $x \neq y \neq z \in \Y$, $d'_{y}(x,z) \leq d_{y}(x,z) \leq d'_{y}(x,z) + 2\theta$.

Fix $K \geq 0$.  Then two vertices $x,z$ of $\P_{K}(\Y)$ are connected by an edge if $d'_{y}(x,z) \leq K$ for all $y \in \Y - \{x,z\}$.  Let $d$ denote the resulting path metric on $\P_{K}(\Y)$.  

Bestvina--Bromberg--Fujiwara showed that for $K$ large enough relative to $\theta$, there are constants $\Ke$, $\Kp$, and $\Kg$,  so that the following properties hold (see \cite[Proposition~3.14~and~Lemma 3.18]{ar:BBF15}):

\medskip \noindent {\bf Bounded edge image.} If $x\neq y \neq z$ are vertices of $\P_K(\Y)$ and $d(x,z)=1$, then $d_{y}(x,z) \leq \Ke$.  

\medskip \noindent {\bf Bounded path image.}  If a path in $\P_K(\Y)$ connects vertices $x$ to $z$ without passing through the 2--neighborhood of the vertex $y$, then $d_y(x,z) \leq \Kp$. 

\medskip \noindent {\bf Bounded geodesic image.} If a geodesic in $\P_K(\Y)$ connects vertices $x$ to $z$ without passing through the vertex $y$, then $d_y(x,z) \leq \Kg$. \medskip

(The bounded edge image property follows from the definition of the edges of $\P_K(\Y)$, with $\Ke = K+2\theta$.)  If $K$ is large enough so that the graph $\P_K(\Y)$ satisfies the bounded edge, path, and geodesic properties for some $\Ke$, $\Kp$, and $\Kg$, then we say that $\P_K(\Y)$ is a \emph{projection complex}.  

This is the same definition as we used in our previous work~\cite{un:CMM}.  As mentioned there, we note that our terminology is not standard; in the papers by Bestvina--Bromberg--Fujiwara~\cite{ar:BBF15} and Bestvina--Bromberg--Fujiwara--Sisto~\cite{ar:BBFS20}, every $\P_K(\Y)$ is called a projection complex.  

\p{Group actions on projection complexes.} We say that a group $G$ acts on a projection complex $\P_K(\Y)$ if $G$ acts on the set $\Y$ in such a way that the associated distance functions $d_y$ are $G$--invariant, i.e., $d_{gy}(gx,gz) = d_{y}(x,z)$.  We note that if the original distance functions $d_y$ are $G$--invariant, then the modified distance functions are $G$--invariant as well---as is evident from the definition~\cite[Definition~3.1]{ar:BBF15}---and so the action of $G$ on $\Y$ extends an action of $G$ on the graph $\P_{K}(\Y)$ by simplicial automorphisms.  

\subsection{Windmills}\label{sec:windmills}

To understand the action of $\grp{R_v}$ on $\P$, in our previous work we used the notion of a windmill.  This tool is also necessary in this current work and we review the construction now.

Given an action of a group $G$ on a projection complex $\P$ with an equivariant family of subgroups $\{R_v\}$ of $G$, we can inductively define a sequence of subgraphs $W_i$ of $\P$, a sequence of subsets $\calO_i$ of the set of vertices of $\P$, and a sequence of subgroups $H_i$ of $G$ as follows.  

Let $v_0$ be some base point for $\P$.  To begin the inductive definitions at $i=0$, we define:
\begin{itemize}
\item $H_0 = R_{v_0}$ and
\item $W_0 = \calO_0 =  \{v_0\}$.
\end{itemize}
For $i \geq 1$, we denote by $N_i$ the 1--neighborhood of $W_{i-1}$, we denote by $L_i$ the vertices of $N_i \setminus W_{i-1}$, and we define:
\begin{itemize}
\item $H_i = \grp{R_v \mid v \in N_{i}}$,
\item $W_i = H_i \cdot N_i$, and
\item $\calO_i =$ a set of orbit representatives for the action of $H_{i-1}$ on $L_i$.
\end{itemize}
The set $\{(H_i,W_i,\calO_i)\}_{i=0}^{\infty}$ is called a set of \emph{windmill data} for the equivariant family $\{R_v\}$.  We observe that each $W_i$ is connected.    

The subgroup $H$ of $G$ generated by the $R_v$ is the direct limit of the $H_i$.  Let $\calO$ be the union of the sets of representatives $\calO_i$.  In our previous work with Margalit~\cite[Theorem~1.6]{un:CMM}, we proved the existence of a constant $L(\P)$ such that if $L \geq L(\P)$ and $\{R_v\}$ is an equivariant $L$--spinning family of subgroups then
\[ H \cong \BigFreeProd{v \in \calO} R_v. \]
For the remainder, we will always assume that $L \geq L(\P)$ whenever we are discussing an equivariant $L$--spinning family so that this free product decomposition is valid.  Each of the constants of the form $L_\ast$ defined in the sequel is at least $L(\P)$.

\subsection{Pivot Points}\label{sec:pivot points}

In our previous work, we introduced the notion of the set of pivot points for an element $h$ of $H$ in order to understand the group structure of $H$~\cite{un:CMM}.  We review this notion now and state Lemma~\ref{lem:pivot point facts} which records the necessary technical facts required for the shortening argument in Section~\ref{sec:shortening}.

The \emph{level} of a nontrivial element $h \in H$ is the minimal index $i$ such that $h \in H_i$.  We define the level of the identity element to be $-1$.  

Each $h \in H$ with level $i$ has a \emph{syllable decomposition} $h_1 \cdots h_n$ where each syllable $h_k$ is either a nontrivial element of $H_{i-1}$ or a nontrivial element of $R_{v_k}$ with $v_k \in \calO_i$.  Moreover no two consecutive syllables are of the first type and consecutive syllables $h_k$ and $h_{k+1}$ of the second type have distinct corresponding fixed vertices $v_k$ and $v_{k+1}$.  We refer to $n$ as the \emph{syllable length} of $h$.  As long as $L \geq L(\P)$, which will be our standing assumption, this syllable decomposition is unique for an equivariant $L$--spinning family $\{ R_v \}$.

Let $i \geq 1$ and fix some element $h$  of $H$ with level $i$ and with syllable decomposition $h=h_1\cdots h_n$.  For $k \in \{1,\dots, n\}$ with $h_k \notin H_{i-1}$ and with corresponding fixed vertex $v_k$ we define a vertex $w_k$ of $\P$ as follows:
\[
w_k = h_1 \cdots h_{k-1}v_k.
\]
Note that $v_k$ and $w_k$ are not defined for the syllables $h_k$ that lie in $H_{i-1}$.  Let $\Piv(h)$ be the ordered list of points $w_k$, and call these the \emph{pivot points} for $h$. For $h \in H_0$ we define $\Piv(h)$ to be empty.

There are several key properties regarding windmills and pivot points that we recall now.

\begin{lemma}\label{lem:pivot point facts}
Let $\P$ be a projection complex and let $G$ be a group acting on $\P$.  There are constants $\Lwind$ and $m$ with the following properties.  Suppose $L \geq \Lwind$ and suppose $\{R_v\}$ is an equivariant $L$--spinning family of subgroups of $G$. Let $H = \grp{R_v}$ and choose windmill data $\{(H_i,W_i,\calO_i)\}$.  
\begin{enumerate}
\item If $h$ is an element of $H$ and if $w \in \Piv(h)$, then $d_{w}(v_0,hv_0) > L/2$.
\item If $h$ is an element of $H$ and if $w$, $w'$ are pivot points for $h$ with $w < w'$, then \[d_{w}(v_0,w') > L/2 - \theta \mbox{ and } d_{w'}(v_0,w) \leq \theta.\]
\item For all $i \geq 1$, if $x \in N_{i}$ and $v \notin W_{i-1}$ with $v \neq x$, then $d_v(v_0,x) \leq m$.
\item For all $i \geq 1$, if $h$ has level $i$, then no pivot point for $h$ lies in $W_{i-1}$.   
\end{enumerate}

\begin{proof}
Using the constants associated with $\P$, we set $m = 11\Ke + 6\Kg + 5\Kp$ and $\Lwind = 4(m + \theta) + 1$.  We remark that $m$ is the same constant the proof of Theorem~1.6 in our prior work~\cite{un:CMM} and that $\Lwind \geq L(\P)$ from that same theorem.  The above listed facts follow from results and arguments appearing in the proof of Theorem~1.6 in that paper as we now explain.  

\medskip \noindent {\it Proof of (1).}  Fix an element $h$ in $H$ with syllable decomposition $h = h_1\cdots h_n$.  If the level of $h$ is less than $1$, the statement is vacuous.   Hence suppose that the level of $h$ is at least $1$.  Consider a pivot point $w = h_1 \cdots h_{k-1}v_k$ for $h$.  Equation (1) in the proof of Theorem~1.6 in our prior work states that
\[ d_{w}(v_0,hv_0) \geq d_{w}(v_0,h_kv_0) - 2(m+\theta). \]  As $d_{w}(v_0,h_kv_0) \geq L$ and $L/2 > 2(m+\theta)$, the statement holds.

\medskip \noindent {\it Proof of (2).} Again, fix an element $h$ in $H$ and assume that the level of $h$ is at least $1$ as the statement is vacuous otherwise.  Let $w$ and $w'$ be pivot points for $h$ with $w < w'$.  Statement (B) of the inductive hypothesis in the proof of Theorem~1.6 implies that there is a geodesic from $v_0$ to $w$ avoiding $w'$.  Hence we have $d_{w'}(v_0,w)  \leq \Kg$. Therefore, using the first item, we have
\[d_{w'}(w,hv_0) \geq d_{w'}(v_0,hv_0) - d_{w'}(v_0,w) > L/2 - \Kg > \theta.\]  
Thus by the Inequality on triples, we find $d_{w}(w',hv_0) \leq \theta$.  From this, using the first item again, we conclude
\[ d_{w}(v_0,w') \geq d_{w}(v_0,hv_0) - d_{w}(w',hv_0) > L/2 - \theta. \]
As $d_w(v_0,w') > L/2-\theta > \theta$, by the Inequality on triples, we have $d_{w'}(v_0,w) \leq \theta$.

\medskip \noindent {\it Proof of (3).} This is statement (C) of the inductive hypothesis in the proof of Theorem~1.6.

\medskip \noindent {\it Proof of (4).} Fix $i \geq 1$ and let $h$ be an element of $H$ with level $i$.  The first pivot point for $h$, $w$, lies in $L_i$ by definition.  As $L_i$ is disjoint from $W_{i-1}$, the statement holds for this pivot point.  Let $w'$ be another pivot point for $h$.  By the second item, we have $d_{w}(v_0,w') > L/2 - \theta > m$.  If $w' \in W_{i-1} \subset N_i$, then as $w \notin W_{i-1}$, the third item would imply that $d_{w}(v_0,w') \leq m$.  This is a contradiction, hence $w' \notin W_{i-1}$.    
\end{proof}
\end{lemma}

\section{Shortening via pivot points}\label{sec:shortening}

In this section we introduce the key technical tool: \emph{shortening}.  The precise statement is given in Proposition~\ref{prop:shorten}.  This proposition will allow us to bend paths in $\P$ without changing their images in the quotient $\P/\grp{R_v}$.  The bent path has a lower complexity in a precise sense that we will explain.  This will allow us to conclude that certain closed paths in $\P/\grp{R_v}$ lift to closed paths in $\P$.  

Before we can state Proposition~\ref{prop:shorten} we need to alter the notions of level and pivot points so that they are better suited for conjugacy classes.  Assume that $\{R_v\}$ is an equivariant $L$--spinning family of subgroups of $G$. Let $H = \grp{R_v}$ and choose windmill data $\{(H_i,W_i,\calO_i)\}$.        

\ip{Complexity of an element in $H$.}  The \emph{complexity} of an element $h \in H$ is the ordered pair $(\ii(h),\nn(h))$ where $\ii(h)$ is the minimal index of any $H$--conjugate of $h$ and $\nn(h)$ is the minimal syllable length of any $H$--conjugate of $h$ that has level $\ii(h)$.  Lexicographical order on the pair $(\ii(h),\nn(h))$ gives a  weak order on the elements in $H$.  The only element with $\ii(h) = -1$ is the trivial element.  Also, we remark that if $\ii(h) = 0$, then $\nn(h) = 1$.

\ip{Essential pivot points.}  Given an element $h \in H$ with $\ii(h) = i$ and $\nn(h) = n$, we can express $h$ as a reduced word
\[h = g (h_{1}\cdots h_{n})g^{-1}\] 
where each $h_k$ is either a nontrivial element of $H_{i-1}$ or a nontrivial element of $R_{v_k}$ with $v_k \in \calO_i$ and $g \in H$.  If $\nn(h) > 1$, then minimality of $\nn(h)$ implies that if $h_1 \in H_{i-1}$ then $h_n \notin H_{i-1}$, and that if $h_1 \in R_{v_1}$, then $h_n \notin R_{v_1}$. The subset of $\Piv(h)$ corresponding to the syllables $h_k$ that lie in $R_{v_k}$ for some $v_k \in \calO_i$ are called \emph{essential pivot points}.  We denote this subset by $\Pive(h)$.  This set is nonempty so long as $\ii(h) \geq 1$.     

\medskip 

The following lemma, whose proof is an easy exercise from the definitions, justifies calling these pivot points essential.

\begin{lemma}\label{lem:essential facts}
Let $\P$ be a projection complex and let $G$ be a group acting on $\P$.  Suppose $\{R_v\}$ is an equivariant $L$--spinning family of subgroups of $G$.  Let $H  = \grp{R_v}$ and choose windmill data $\{(H_i,W_i,\calO_i)\}$.  The following statements are true.
\begin{enumerate}
\item If the syllable length of $h \in H$ equals $\nn(h)$, then every pivot point is essential, i.e., $\Piv^*(h) = \Piv(h)$.
\item If $h$ and $g$ are elements of $H$, then $\Pive(ghg^{-1}) = g\Pive(h)$.
\item If $h$ is an element of $H$ and $k \geq 1$, then \[\Pive(h^k) = \bigcup_{j=0}^{k-1} h^j\Pive(h)\] as ordered sets. 
\end{enumerate}  
\end{lemma}

Items (1) and (2) imply that if the syllable length of $h$ equals $\nn(h)$, then $\Piv^*(ghg^{-1}) = g\Piv(h)$.  We remark that items (2) and (3) of Lemma~\ref{lem:essential facts} are false for the set of all pivot points.  We now state and prove the shortening proposition.

\begin{proposition}\label{prop:shorten}
Let $\P$ be a projection complex and let $G$ be a group acting on $\P$. There is a constant $\Lshort$ with the following properties.  Suppose $L \geq \Lshort$ and suppose $\{R_v\}$ is an equivariant $L$--spinning family of subgroups of $G$.  Let $H = \grp{R_v}$ and choose windmill data $\{(H_i,W_i,\calO_i)\}$.  Let $x$ be a vertex in $\P$ and $h \in H$ such that $hx \neq x$.  Then there exists a vertex $v$ of $\P$ and element $h_{v}$ of $R_{v}$ such that
\begin{enumerate}

\item either $v \in \{x,hx\}$ or $d_{v}(x,hx) > L/10$; and

\item $h_vh < h$.

\end{enumerate}
\end{proposition}

The first item roughly translates as stating that $v$ lies on the geodesic from $x$ to $hx$.

\begin{proof}
Let $\Lwind$ and $m$ be the constants from Lemma~\ref{lem:pivot point facts}.  Set $\Lshort = \max\{\Lwind,5m,14\theta\}$.  Take $L \geq \Lshort$ and suppose that $G$ is acting on $\P$ and that $\{R_v\}$ is an equivariant $L$--spinning family of subgroups of $G$.

Fix a vertex $x$ of $\P$ and an element $h$ of $H$ such that $hx \neq x$.  Let $i = \ii(h)$, $n = \nn(h)$ and express $h$ as a reduced word
\[h = gh_{1}\cdots h_{n}g^{-1}\] where each $h_k$ is either a nontrivial element of $H_{i-1}$ or a nontrivial element of $R_{v_k}$ with $v_k \in \calO_i$.  

First, suppose that $i = 0$ and so $h = gh_1g^{-1}$ where $h_1 \in R_{v_0}$.  In this case, we take $v = gv_0$ and $h_v = gh_1^{-1}g^{-1} \in R_v$.  If $v \notin \{x, hx\}$, then $d_v(x,hx) = d_{v_0}(g^{-1}x,h_1g^{-1}x) \geq L > L/10$.  As $h_v h$ is the identity, clearly $h_v h < h$.  

Hence for the remainder, we assume that $i$ is at least $1$.  In particular, the set $\Pive(h)$ is nonempty.  Our strategy is to find an essential pivot point $w$ for $h$ and an integer $p$ such that $v = h^{p}w$ satisfies the first item.  Given such a pivot point $w = gh_\sigma v_k$, where $h_\sigma = h_1\cdots h_{k-1}$, we take $h_{v} = h^{p}(gh_{\sigma}) h_{k}^{-1}(gh_{\sigma})^{-1}h^{-p} \in R_{v}$.  Then 
\begin{align*}
h_{v}h &= \bigl(h^{p}(gh_{\sigma}) h_{k}^{-1}(gh_{\sigma})^{-1}h^{-p}\bigr)h \\
&= h^{p}\bigl((gh_{\sigma})h_{k}^{-1}(gh_{\sigma})^{-1}h\bigr)h^{-p} \\
&= h^{p}(gh_{1}\cdots h_{k-1}h_{k+1} \cdots h_{n}g^{-1})h^{-p}.
\end{align*}
Hence for this element we have $h_vh < h$, which is the second item.        

If $\{x,hx\} \cap \Pive(h) \neq \emptyset$, we can take $w$ to be an essential pivot point in this intersection and set $v = w$.   

Thus we may assume that $\{ x, hx \} \cap \Pive(h) = \emptyset$.  There are two cases depending on whether $x \in gW_{i}$ or $x \notin gW_{i}$.  Set $\bar{h} = h_1\cdots h_n$ so that $h = g\bar{h}g^{-1}$.  We observe that $\bar{h}$ has level $i$. 

For the first case, we initially assume that $x \in gN_{i} \subset gW_{i}$.  Let $w$ be a pivot point for $\bar{h}$, thus $gw$ is an essential pivot point for $h$.  By Lemma~\ref{lem:pivot point facts}(1), we have that $d_{w}(v_{0},\bar{h}v_{0}) > L/2$.  By Lemma~\ref{lem:pivot point facts}(4), we have that $w \notin W_{i-1}$.  Since $\bar{h}^{-1}w$ is a pivot point for $\bar{h}^{-1}$, Lemma~\ref{lem:pivot point facts}(4) also implies that $\bar{h}^{-1}w \notin W_{i-1}$ as well.  Hence by Lemma~\ref{lem:pivot point facts}(3) as $g^{-1}x \in N_{i}$ and $w, \bar{h}^{-1}w \notin W_{i-1}$ we have that $d_{w}(g^{-1}x,v_{0}) \leq m$ and $d_{w}(\bar{h}g^{-1}x,\bar{h}v_{0}) = d_{\bar{h}^{-1}w}(g^{-1}x,v_{0}) \leq m$.  Therefore
\[ d_{gw}(x,hx) = d_{w}(g^{-1}x,\bar{h}g^{-1}x) \geq d_{w}(v_{0},\bar{h}v_{0}) - d_{w}(v_{0},g^{-1}x) - d_{w}(\bar{h}v_{0},\bar{h}g^{-1}x) > L/2 - 2m \geq L/10. \]     
Hence we may set $v = gw$.

Now suppose that $x \in gW_{i} - gN_{i}$.  Then there is an $h_0 \in H_i$ such that $h_0x \in gN_{i}$.  Let $h' = h_0 h h_0^{-1}$ and $x' = h_0 x$.  We have $h'x' \neq x'$.  Fix some pivot point $w$ for $\bar{h}$ and so $gw$ is an essential pivot point  for $h$.  By Lemma~\ref{lem:essential facts}(1), we have $h_0gw \in \Pive(h')$.  As $x,hx \notin \Pive(h)$, we have that $h_0gw \neq x',h'x'$.  Thus as $x' \in gN_{i}$, the above case applies and we have that
\[ d_{gw}(x,hx) = d_{h_0gw}(x',h'x') > L/10. \]
Hence we may set $v = gw$.

Lastly, we deal with the second case that $x \notin gW_{i}$.  In this case, we will be considering the projection of $x$ to various points of the form $h^jw$ where $w$ is an essential pivot point for $h$ and $j$ is an integer.  As $w$ lies in $gW_i$ by definition and $W_i$ is $H_i$--invariant, we have that $h^jw$ lies in $gW_{i}$.  In particular, $x \neq h^jw$ for any essential pivot point $w$ for $h$ and any integer and therefore projections of $x$ to such points are always defined.    

Fix any essential pivot point $w$ for $h$.  By Lemma~\ref{lem:essential facts}(2) we have that $h^{j}w$ is an essential pivot point for $h^{k}$ whenever $0 \leq j < k$ and additionally, such points are ordered $h^{j_1}w < h^{j_2}w$ if $j_1 < j_2$.  By Lemma~\ref{lem:pivot point facts}(2), we have that for $1 \leq j_{1} < j_{2}$ that
\[d_{h^{j_{1}}w}(w,h^{j_{2}}w) \geq d_{h^{j_{1}}w}(v_{0},h^{j_{2}}w) - d_{h^{j_{1}}w}(v_{0},w) \geq L/2 - 2\theta.\]
By a similar argument we have $d_{h^{j_{1}}w}(h^{j_{0}}w,h^{j_{2}}w) \geq L/2 - 2\theta$ for all integers $j_{0} < j_{1} < j_{2}$.

\medskip \noindent {\it Claim.} There is an integer $J$ such that $d_{h^{j}w}(h^{j-1}w,x) > \theta$ for $j \leq J$ and $d_{h^{j}w}(h^{j-1}w,x) \leq \theta$ for all $j > J$. 

\medskip

\noindent We first show that the set $\{ j \in \Z \mid d_{h^{j}w}(h^{j-1}w,x) \leq \theta \}$ has the form $(J,+\infty)$ for some $J \in \Z \cup \{-\infty,+\infty\}$.  To this end, we suppose that $d_{h^{j}w}(h^{j-1}w,x) \leq \theta$.  If $d_{h^{j+1}w}(h^{j}w,x) > \theta$, then by the Inequality on triples we have $d_{h^{j}w}(h^{j+1}w,x) \leq \theta$.  In this case we find that
\[ L/2 - 2\theta \leq d_{h^{j}w}(h^{j-1}w,h^{j+1}w) \leq d_{h^{j}w}(h^{j-1}w,x) + d_{h^{j}w}(h^{j+1}w,x) \leq 2\theta. \]
This is a contradiction as $L > 8\theta$ and therefore $d_{h^{j+1}w}(h^{j}w,x) \leq \theta$ too.  

If $J = -\infty$, then $d_{h^jw}(h^{j-1}w,x) \leq \theta$ for all integers $j$.  Thus for all $j \leq -1$ we find that
\begin{align*}
d_{h^{j}w}(w,x) \geq d_{h^{j}w}(h^{j-1}w,w) - d_{h^{j}w}(h^{j-1}w,x) \geq L/2 - 3\theta > \theta.
\end{align*}
This contradicts the Finiteness axiom.

If $J = +\infty$, then by the Inequality on Triples we have $d_{h^{j}w}(h^{j+1}w,x) \leq \theta$ for all integers $j$.  Thus for all $j \geq 1$ we find that
\begin{align*}
d_{h^{j}w}(w,x) \geq d_{h^{j}w}(h^{j+1}w,w) - d_{h^{j}w}(h^{j+1}w,x) \geq L/2 - 3\theta > \theta.
\end{align*}
Again, this contradicts the Finiteness axiom.  This completes the proof of the claim.

\medskip 

Let $J$ be as defined in the Claim.  To complete the proof of the proposition, there are two cases based on $d_{h^Jw}(h^{J-1}w,x)$. We will show that we can take $v$ to be either $h^Jw$ or $h^{J+1}w$.

First, suppose that $d_{h^{J}w}(h^{J-1}w,x) \leq L/4$.  We have $d_{h^{J}w}(h^{J-1}w,x) > \theta$ and by the Inequality on triples and invariance we have $d_{h^{J}w}(h^{J+1}w,hx) = d_{h^{J-1}w}(h^{J}w,x) \leq \theta$.
Thus
\begin{align*} 
d_{h^{J}w}(x,hx) &\geq d_{h^{J}w}(h^{J-1}w,h^{J+1}w) - d_{h^{J}w}(h^{J-1}w,x) - d_{h^{J}w}(h^{J+1}w,hx) \\
& \geq L/2 - \theta - L/4 - \theta \geq L/4 - 2\theta > L/10.
\end{align*}
Hence we can set $v = h^{J}w$.

Else, we have that $d_{h^{J}w}(h^{J-1}w,x) = d_{h^{J+1}w}(h^{J}w,hx) > L/4$.  As $d_{h^{J+1}}(h^Jw,x) \leq \theta$ we have
\begin{align*} 
d_{h^{J+1}w}(x,hx) &\geq d_{h^{J+1}w}(h^{J}w,hx) - d_{h^{J+1}w}(h^{J}w,x) \\
& \geq L/4 - \theta > L/10.
\end{align*}
Hence we can set $v = h^{J+1}w$.
\end{proof}

\section{Lifting Quadrilaterals}\label{sec:lifting}

In this section, we apply the shortening argument of Proposition~\ref{prop:shorten} to show that geodesic quadrilaterals in the quotient of the projection complex $\P/\grp{R_v}$ lift to geodesic quadrilaterals in the projection complex $\P$.  This is stated in Proposition~\ref{prop:lift quad}.  As mentioned in the Introduction, the strategy to show that $\P/\grp{R_v}$ is $\delta$--hyperbolic is to lift geodesic triangles in $\P/\grp{R_v}$ to geodesic triangles in $\P$.  As a triangle is a degenerate quadrilateral where one side has length 0, Proposition~\ref{prop:lift quad} applies to geodesic triangles as well.  The reason we work with quadrilaterals is to show that the action of $G/\grp{R_v}$ on $\P/\grp{R_v}$ is a non-elementary WPD action, so long as the action of $G$ on $\P$ is and $L$, the spinning constant, is large enough.

There are two items we need to discuss before stating and proving Proposition~\ref{prop:lift quad}.

\ip{Lifting geodesics.} Throughout this section we will be lifting geodesics from $\P/\grp{R_v}$ to $\P$ and modifying the lifts.  It will be important to have a way of certifying that these lifts and their modifications are geodesics.  This is the content of the following lemma.  Throughout the rest of the paper, we will always assume that paths are 1--Lipschitz. 

\begin{lemma}\label{lem:lift geo}
Let $\P$ be a projection complex and let $G$ be a group acting on $\P$.  Suppose that $H$ is a subgroup of $G$ and let $p \from \P \to \P/H$ be the quotient map.  The following statements are true.
\begin{enumerate}
\item If $\bar{\alpha} \from [0,n] \to \P/H$ is a path and $x$ is a point in $\P$ that satisfies $p(x) = \bar{\alpha}(0)$, then there exists a path $\alpha \from [0,n] \to \P$ such that $p \circ \alpha = \bar{\alpha}$ and $\alpha(0) = x$.  
\item If $\alpha \from [0,n] \to \P$ is a path and $n = d_{\P/H}(p(\alpha(0)),p(\alpha(n)))$, then $\alpha$ is a geodesic.
\end{enumerate} 
\end{lemma}

\begin{proof}
The first statement is obvious.

The second statement follows as the map $p \from \P \to \P/H$ is 1--Lipschitz.  Indeed, if $\alpha$ is not a geodesic, then there is a geodesic $\alpha' \from [0,n'] \to \P$ where $\alpha'(0) = \alpha(0)$, $\alpha'(n') = \alpha(n)$ and $n' < n$.  As $p$ is 1--Lipschitz, we find
\[ n = d_{\P/H}(p(\alpha'(0)),p(\alpha'(n'))) \leq n'. \]  This a contradiction and hence $\alpha$ is a geodesic.
\end{proof}

\ip{Bending paths.}  Let $v$ be a vertex in $\P$.  Suppose $\alpha \from [0,n] \to \P$ is a path and that $v = \alpha(n_0)$ for some $n_0 \in \{0,\ldots,n\}$.  For any $h_v \in R_v$ we define a new path $\alpha \vee_v h_v \from [0,n] \to \P$ by
\begin{equation*}
\bigl(\alpha \vee_v h_v\bigr)(t) = \begin{cases}
\alpha(t) & \mbox{ if } 0 \leq t \leq n_0, \mbox{ or} \\
h_v\alpha(t) & \mbox{ if } n_0 \leq t \leq n.
\end{cases}
\end{equation*}
As $\alpha(n_0) = h_v\alpha(n_0)$, this does define a path.  We say that $\alpha \vee_v h_v$ is obtained by \emph{bending $\alpha$ at $v$ using $h_v$}.  Writing $\alpha$ as the concatenation of two paths $\alpha_1$ and $\alpha_2$ where $\alpha_1$ ends at $v$ and $\alpha_2$ begins at $v$, the bent path $\alpha \vee_v h_v$ is the concatenation of $\alpha_1$ and $h_v\alpha_2$.  See Figure~\ref{fig:bend}.

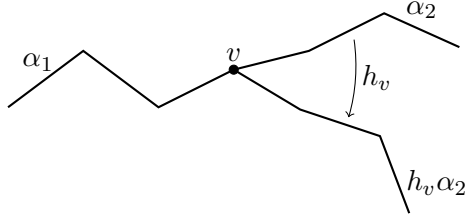
\begin{figure}[ht]
\centering
\begin{tikzpicture}
\draw[thick] (-3,-0.5) -- (-2,0.25) -- (-1,-0.5) -- (0,0);
\draw[thick] (0,0) -- (1,0.25) -- (2,0.75) -- (3,0.3);
\draw[thick,rotate=-45] (0,0) -- (1,0.25) -- (2,0.75) -- (3,0.3);
\draw[->] (1.6,0.4) arc (10:-20:1.6cm and 2cm);
\filldraw (0,0) circle (0.06);
\node at (0,0.2) {$v$}; 
\node at (-2.6,0.1) {$\alpha_1$};
\node at (2.5,0.8) {$\alpha_2$};
\node at (2.7,-1.5) {$h_v\alpha_2$};
\node at (1.9,-0.1) {$h_v$};
\end{tikzpicture}
\caption{The paths $\alpha$ and $\alpha \vee_v h_v$}\label{fig:bend}
\end{figure}

\begin{lemma}\label{lem:bending}
Let $\P$ be a projection complex and let $G$ be a group acting on $\P$. Suppose $\{R_v\}$ is an equivariant family of subgroups of $G$.  Let $H = \grp{R_v}$ and let $p \from \P \to \P/H$ be the quotient map.  Let $\alpha \from [0,n] \to \P$ be a path and let $v$ be a vertex in the image of $\alpha$.  Then for any $h_v \in R_v$  the following statements are true.
\begin{enumerate}
\item We have $p \circ \alpha = p \circ \bigl(\alpha \vee_v h_v \bigr)$.
 
\item For any $0 \leq t_1 < t_2 \leq n$, if $p \circ \alpha|[t_1,t_2]$ is a geodesic, then so is $\bigl(\alpha \vee_v h_v \bigr)|[t_1,t_2]$.
\end{enumerate} 
\end{lemma}

\begin{proof}
The first statement follows immediately from the definitions.

The second statement follows from the first statement and Lemma~\ref{lem:lift geo}(2).
\end{proof}

Let $X$ be a graph considered as a metric space where every edge has length one. 
A \emph{geodesic quadrilateral} $Q$ in $X$ consists of four geodesics and four points: $\alpha_k$ from $x_k$ to $x_{k+1 \mod 4}$ for $k = 0,1,2,3$.  We write $Q = \cup_{k=0}^3 \alpha_k$.

\begin{proposition}\label{prop:lift quad}
Let $\P$ be a projection complex and let $G$ be a group acting on $\P$.  For any $B \geq 0$, there is a constant $\Llift(B)$ with the following properties.  Suppose $L \geq \Llift(B)$ and suppose $\{R_v\}$ is an equivariant $L$--spinning family of subgroups of $G$.  Let $H = \grp{R_v}$ and let $p \from \P \to \P/H$ be the quotient map.  For each geodesic quadrilateral $\bar{Q} = \cup_{k=0}^3 \bar{\alpha}_k$ in $\P/H$ there exists a geodesic quadrilateral $Q = \cup_{k=0}^3 \alpha_k$ in $\P$ so that $p(\alpha_k) = \bar{\alpha}_k$ for $k = 0,1,2,3$.  

Additionally, $Q$ satisfies the following property.  If there are lifts $\tilde\alpha_0$ from $\tilde x_0$ to $\tilde x_1$ and $\tilde\alpha_2$ from $\tilde x_2$ to $\tilde x_3$ of $\bar{\alpha}_0$ and $\bar{\alpha}_2$ respectively such that $d_v(\tilde x_0,\tilde x_1) \leq B$ and $d_v(\tilde x_2,\tilde x_3) \leq B$ when defined, then the geodesics $\alpha_0$ and $\alpha_2$ in $Q$ are $H$--translates of $\tilde\alpha_0$ and $\tilde \alpha_2$ respectively.
\end{proposition}

\begin{proof}
Fix $B \geq 0$ and set $\Llift(B) = \max\{\Lshort,40B,40\Kg\}$.  Take $L \geq \Llift(B)$ and suppose that $G$ is acting on $\P$ and that $\{R_v\}$ is an equivariant $L$--spinning family of subgroups of $G$.

Let $\bar{x}_{0}, \bar{x}_{1}, \bar{x}_{2}$ and $\bar{x}_{3}$ be the vertices of the geodesic quadrilateral $\bar{Q}$ in $\P/H$.  By Lemma~\ref{lem:lift geo}(1), for any point $x_0 \in p^{-1}(\bar{x}_0)$, we can iteratively lift the geodesics $\bar{\alpha}_k$ to paths $\alpha_k$ from $x_k$ to $x_{k+1}$ where $p(x_k) = \bar{x}_k$.  By Lemma~\ref{lem:lift geo}(2), the paths $\alpha_k$ are geodesics.  If $\alpha_0$ and $\alpha_2$ as in the statement of the proposition exists, then we can ensure that $\alpha_0$ and $\alpha_2$ are $H$--translates of these geodesics.  We denote the concatenation of the paths $\alpha_k$ by $\alpha$ and we say that $\alpha$ is a \emph{special lift} of $\bar{Q}$.  

For each special lift $\alpha$ of $\bar{Q}$, with endpoints denoted $x_0$ and $x_4$, there is an element $\hh(\alpha) \in H$ with minimal complexity such that $x_{4} = \hh(\alpha)x_{0}$.  Let $\alpha$ be a special lift of $\bar{Q}$ so that $\hh(\alpha)$ has minimal complexity among all special lifts of $\bar{Q}$.  

We claim that $x_0 = x_4$, which shows that $\alpha$ defines a geodesic quadrilateral $Q$ as in the statement of the proposition.  Indeed, if not we will show that we can bend $\alpha$ to a new path $\alpha'$ that is a special lift with $\hh(\alpha') < \hh(\alpha)$.  This contradicts the minimality of $\hh(\alpha)$.

To this end, suppose that $x_0 \neq x_4 = \hh(\alpha)x_0$.  Apply Proposition~\ref{prop:shorten} to $x = x_0$ and $h = \hh(\alpha)$ and let $v$ be the corresponding vertex of $\P$ and $h_v \in R_v$ the corresponding element.  We have that $h_v\hh(\alpha) < \hh(\alpha)$.

We claim that $v$ lies in the image of $\alpha$.  Indeed, if $v \notin \{x_0,x_4\}$, then $d_v(x_0,x_4) > L/10$.  If further that $v \notin \{ x_1,x_2,x_3\}$, then by the triangle inequality, we have that $d_v(x_n,x_{n+1}) > L/40$ for some $n$.  As $L/40 \geq \Kg$, there is a $n_0$ such that $\alpha_n(n_0) = v$.  Moreover, as $L/40 \geq B$, if lifts $\tilde\alpha_0$ and $\tilde\alpha_2$ as in the statement of the proposition exists, we must have that $n = 1$ or $n = 3$.  This shows that $v$ lies in the image of $\alpha$.  We consider the path $\alpha' = \alpha \vee_v h_v$.     
  
By Lemma~\ref{lem:bending}(2), $\alpha'$ consists of four geodesic segments $\alpha'_k$ for $k = 0,1,2,3$.  Moreover, we observe that $\alpha'$ is a special lift of $\bar{Q}$ as if lifts $\tilde\alpha_0$ and $\tilde\alpha_2$ as in the statement of the proposition exists, then the segments $\alpha'_0$ and $\alpha'_2$ are $H$--translates the segments $\alpha_0$ and $\alpha_2$ respectively.  Letting $x'_4$ denote the terminal point of $\alpha'$ we find  \[ x'_4 = h_vx_4 = h_v\hh(\alpha)x_0 \] so that $\hh(\alpha') \leq h_v\hh(\alpha) < \hh(\alpha)$.  This contradicts the minimality of $\hh(\alpha)$.
\end{proof}

\section{Proof of Theorem~\ref{thm:hyperbolic quotient}}\label{sec:proof hyperbolic}

In this section we prove the first of the two main results of this paper.  Theorem~\ref{thm:hyperbolic quotient} states that if a group $G$ acts on a projection complex $\P$ then there exists a constant $\Lhyp(\P)$ so that if $L \geq \Lhyp(\P)$ and if $\{R_v\}$ is an equivariant $L$--spinning family of subgroups of $G$ then $\P/\grp{R_v}$ is $\delta$--hyperbolic.  The proof proceeds by showing that geodesic triangles in $\P/\grp{R_v}$ can be lifted to geodesic triangles in $\P$.  

\begin{proof}[Proof of Theorem~\ref{thm:hyperbolic quotient}]
Let $\P$ be a projection complex and set $\Lhyp(\P) = \Llift(0)$.  Bestvina--Bromberg--Fujiwara proved the $\P$ is a quasi-tree~\cite[Theorem~3.16]{ar:BBF15}.  Let $\delta$ be such that $\P$ is $\delta$--hyperbolic.  Take $L \geq \Lhyp(\P)$ and suppose that $G$ is acting on $\P$ and that $\{R_v\}$ is an equivariant $L$--spinning family.  Let $H = \grp{R_v}$.  

Let $\bar{\alpha}_0$, $\bar{\alpha}_1$ and $\bar{\alpha}_2$ be the three sides of a geodesic triangle in $\P/H$.  We set $\bar{\alpha}_3$ to be the trivial path at the endpoint of $\bar{\alpha}_2$.  This gives a (degenerate) geodesic quadrilateral $\bar{Q} = \cup_{k=0}^3 \bar{\alpha}_k$.  By Proposition~\ref{prop:lift quad}, there is a geodesic quadrilateral $Q = \cup_{k=0}^3 \alpha_k$ so that $p(\alpha_k) = \bar{\alpha}_k$ for $k = 0,1,2,3$.  As $\alpha_3$ is a trivial path, $Q$ is in fact a geodesic triangle in $\P$.  

As the map $p \from \P \to \P/H$ is 1--Lipschitz and as $Q$ is $\delta$--thin, the geodesic triangle $\bar{Q}$ is $\delta$--thin as well.  Hence $\P/H$ is $\delta$--hyperbolic.
\end{proof}

\section{Bounded projections}\label{sec:small}

There are two key results in this section.  First, we show that geodesics $\alpha \from [0,n] \to \P$ with bounded projections are mapped by $p$ to geodesics in $\P/\grp{R_v}$.  This appears as Lemma~\ref{lem:project geo}.  The proof of this lemma is very similar to the proof of Proposition~\ref{prop:lift quad} as it involves bending and shortening.  Secondly, we apply Lemma~\ref{lem:project geo} to show that given a WPD element in $G$ where some the orbit os some point has bounded projections, its image in $G/\grp{R_v}$ acts as a WPD element on $\P/\grp{R_v}$.  This appears as Lemma~\ref{lem:still wpd}.  The proof of this lemma uses Proposition~\ref{prop:lift quad}.

\begin{lemma}\label{lem:project geo}
Let $\P$ be a projection complex and let $G$ be a group acting on $\P$.  For any $B \geq 0$, there is a constant $\Lpro(B)$ with the following property.  Suppose $L \geq \Lpro(B)$ and suppose $\{R_v\}$ is an equivariant $L$--spinning family of subgroups of $G$.  Let $H = \grp{R_v}$ and let $p \from \P \to \P/H$ be the quotient map.  If $\alpha \from [0,n] \to \P$ is a geodesic, and $d_v(\alpha(0),\alpha(n))) \leq B$ for all vertices $v$ of $\P$ other than $\alpha(0)$ and $\alpha(n)$, then $p \circ \alpha \from [0,n] \to \P/H$ is a geodesic.
\end{lemma}

\begin{proof}
Set $\Lpro(B) = \max\{\Lshort,10B + 10\Kg\}$.  Take $L \geq \Lpro(B)$ and suppose that $G$ is acting on $\P$ and that $\{R_v\}$ is an equivariant $L$--spinning family.

Let $\bar{\beta} \from [0,n'] \to \P/H$ be a geodesic from $p(\alpha(0))$ to $p(\alpha(n))$.  We will argue that $n = n'$, showing that $p \circ \alpha$ is a geodesic.  

For each $H$--translate $h\alpha \from [0,n] \to \P$ of $\alpha$, we say a lift $\beta\from [0,n'] \to \P$ of $\bar{\beta}$ is \emph{compatible} with $h\alpha$ if $h\alpha(0) = \beta(0)$.  In this situation, there is an element $\hh(h\alpha,\beta)$ with minimal complexity such that $\beta(n') = \hh(h\alpha,\beta)h\alpha(n)$.  We replace $\alpha$ by an $H$--translate and let $\beta \from [0,n'] \to \P$ be a compatible lift of $\bar{\beta}$ so that $\hh(\alpha,\beta)$ minimizes complexity among all $H$--translates of $\alpha$ and compatible lifts.  

We claim that $\alpha(n) = \beta(n')$, which shows that $n = n'$ as both $\alpha$ and $\beta$ are geodesics.  Indeed, if not we will show that we can find a translate $\alpha'$ of $\alpha$ and a compatible lift $\beta'$ with $\hh(\alpha',\beta') < \hh(\alpha,\beta)$.  The path $\beta'$ is obtained by translating or bending $\beta$.  This contradicts the minimality of $\hh(\alpha,\beta)$.  

To this end, suppose that $\alpha(n) \neq \beta(n')$.  Apply Proposition~\ref{prop:shorten} to $x = \alpha(n)$ and $h = \hh(\alpha,\beta)$ and let $v$ be the corresponding vertex and $h_v \in R_v$ the corresponding element.  We have that $h_v \hh(\alpha,\beta) < \hh(\alpha,\beta)$.

There are two cases now depending on $v$. 

If $v = \alpha(n)$, then for the $H$--translate $h_v\alpha$ and compatible lift $h_v\beta$, we have \[h_v\beta(n') = h_v\hh(\alpha,\beta)\alpha(n) = h_v\hh(\alpha,\beta)h_v\alpha(n)\] so that $\hh(h_v\alpha,h_v\beta) \leq h_v\hh(\alpha,\beta) < \hh(\alpha,\beta)$.  This contradicts the minimality of $\hh(\alpha,\beta)$.

Else, we claim that $v$ lies in the image of $\beta$.  Indeed, if $v \neq \beta(n)$ then $d_v(\alpha(n),\beta(n')) > L/10$.  If further $v \neq \beta(0)$, then as $d_v(\alpha(0),\alpha(n)) \leq B$ and $\alpha(0) = \beta(0)$, have have that
\[ d_v(\beta(0),\beta(n')) \geq d_v(\alpha(n),\beta(n')) - d_v(\alpha(0),\alpha(n)) > L/10 - B > \Kg. \]
This shows that $v$ lies in the image of $\beta$.

We define $\beta' = \beta \vee_v h_v$.  By Lemma~\ref{lem:bending}, $\beta'$ is a compatible lift.  Next, we find that
\[ \beta'(n') = h_v\beta(n') = h_v\hh(\alpha,\beta)\alpha(n) \] so that $\hh(\alpha,\beta') \leq \hh(\alpha,\beta) < \hh(\alpha,\beta)$.  This contradicts the minimality of $\hh(\alpha,\beta)$. 
\end{proof}

\begin{lemma}\label{lem:still wpd}
Let $\P$ be a projection complex, let $G$ be group acting on $\P$ and let $B \geq 0$.  Suppose $L \geq \max\{\Llift(B),\Lpro(B)\}$ and suppose $\{R_v\}$ is an equivariant $L$--spinning family of subgroups of $G$.  Let $H = \grp{R_v}$.  If $f \in G$ is a hyperbolic isometry of $\P$ so that $d_v(x_0,f^nx_0) \leq B$ for all $n \in \Z$ when defined, then its image $\bar{f} \in G/H$ is a hyperbolic isometry of $\P/H$.  Additionally, if $f$ is a WPD element, then so is $\bar{f}$.   
\end{lemma}

\begin{proof}
Fix $B \geq 0$ and suppose that $G$ is acting on $\P$ and that $\{R_v\}$ is an equivariant $L$--spinning family where $L \geq \max\{\Llift(B),\Lpro(B)\}$.  Suppose that $f \in G$ is a hyperbolic isometry of $\P$ and $x_0$ is a vertex of $\P$ so that $d_v(x_0,f^nx_0) \leq B$ for all $n \in \Z$ when defined.
 
Let $\bar{x}_0 = p(x_0)$.  As $L \geq \Lpro(B)$, by Lemma~\ref{lem:project geo}, we have that $d_{\P/H}(\bar{x}_0,\bar{f}^n\bar{x}_0) = d_{\P}(x_0,f^nx_0)$.  Hence as $f$ is hyperbolic, $\bar{f}$ is also hyperbolic.

Now assume further that $f$ is a WPD element.  Fix $D \geq 0$ and let $M \geq 0$ be such that the set
\[ \{ g \in G \mid d_{\P}(x_0,gx_0) \leq D \mbox{ and } d_{\P}(f^Mx_0,gf^Mx_0) \leq D  \} \] is finite.  Let $K$ denote the cardinality of this set.

Suppose that $\{\bar{g}_1, \ldots, \bar{g}_{K'} \}$ is a set of elements of $G/H$ so that 
\[d_{\P/H}(\bar{x}_0,\bar{g}_j\bar{x}_0) \leq D \mbox{ and } d_{\P/H}(\bar{f}^M\bar{x}_0,\bar{g}_j\bar{f}^M\bar{x}_0) \leq D.\]
Fix elements $g_j \in G$ whose images are the $\bar{g}_j$s.

We consider the geodesic quadrilateral $\bar{Q}_j = \cup_{k=0}^3 \bar{\alpha}_k$ where: $\bar{\alpha}_0$ is a geodesic from $\bar{x}_0$ to $\bar{f}^M\bar{x}_0$, $\bar{\alpha}_1$ is a geodesic from $\bar{f}^M\bar{x}_0$ to $\bar{g}_j\bar{f}^M\bar{x}_0$, $\bar{\alpha}_2$ is a geodesic from $\bar{g}_j\bar{f}^M\bar{x}_0$ to $\bar{g}_j\bar{x}_0$, and $\bar{\alpha}_3$ is a geodesic from $\bar{g}_j\bar{x}_0$ to $\bar{x}_0$.    

As $L \geq \Llift(B)$ for each $1 \leq j \leq K'$, there is a geodesic quadrilateral $Q_j = \cup_{k=0}^3 \alpha_k$ so that $p(\alpha_k) = \bar{\alpha}_k$.  Moreover, there are elements $h_0,h_2 \in H$ such that $\alpha_0$ is a geodesic from $h_0\bar{x}_0$ to $h_0f^Mx_0$ and $\alpha_2$ is a geodesic from $h_2g_jf^Mx_0$ to $h_2g_jx_0$.  In particular for each $1 \leq j \leq K'$ we find that
\[d_{\P}(x_0,h_0^{-1}h_2g_j) \leq D \mbox{ and } d_{\P}(f^Mx_0,h_0^{-1}h_2g_jf^M x_0) \leq D. \]
This shows that $K' \leq K$.  

As it suffices to check finiteness at a single point, this shows that $\bar{f}$ is a WPD element.  
\end{proof}

\section{Proof of Theorem~\ref{thm:wpd action on quotient}}\label{sec:proof wpd action}

In this section we give the proof of the second of the main results in this paper.  
Theorem~\ref{thm:wpd action on quotient} states that if a group $G$ admits a non-elementary WPD action on a projection complex $\P$ then there exists a constant $\Lwpd(\P,G)$ so that if $L \geq \Lwpd(\P,G)$ and if $\{R_v\}$ is an equivariant $L$--spinning family of subgroups of $G$ then the action of $G/\grp{R_v}$ on $\P/\grp{R_v}$ is a non-elementary WPD action.

\ip{Isometries have bounded projections.} In order to apply the results of Section~\ref{sec:small}, we need to know that hyperbolic isometries of a projection complex have bounded projections.  This is an application of the Finiteness axiom of a projection complex as we now show.

\begin{lemma}\label{lem:bounded projections}
Let $\P$ be a projection complex and let $f$ be a hyperbolic isometry of $\P$.  Then for any vertex $x_0$ of $\P$, there is a constant $B_f$ such that $d_v(x_0,f^n x_0) \leq B_f$ for all $n \in \Z$ when defined.
\end{lemma}

\begin{proof}
Let $M_1 = \max\{ d_v(x_0,f x_0) \mid v \notin \{x_0,f x_0\} \}$ and $M_2 = d_{x_0}(f^{-1}x_0,f x_0)$.  We remark that $M_1$ is finite by the Finiteness axiom.  Set $M = \max\{M_1,M_2 \}$.  Fix a geodesic $\alpha$ from $x_0$ to $f x_0$.  Let $N$ be such that $d(x,f^n y) > 4$ if $x$ and $y$ lie on $\alpha$, and $n \geq N$.  Define $B_f = NM + 2\Kp$.

By equivariance, it suffices to prove the lemma for non-negative integers.  Fix an $n \in \N$ and suppose that $v \notin \{x_0,f^n x_0\}$.  If $v$ does not lie in the 2--neighborhood of the path $\alpha \cup f \alpha \cup \cdots \cup f^{n-1}\alpha$, then $d_v(x_0,f^nx_0) \leq \Kp \leq B_f$.

Else, there are indices $0 \leq i_0 \leq i_1 \leq n-1$ such that $i_1 - i_0 
< N$ and $v$ lies in the 2--neighborhood of $f^j \alpha$ only if $i_0 \leq j \leq i_1$.  Thus as $v$ does not lie in the 2--neighborhood of $\alpha \cup \cdots \cup f^{i_0-1}\alpha$ nor in the 2--neighborhood of $f^{i_1 + 1}\alpha \cup \cdots \cup f^n\alpha$, we have $d_v(x_0,f^{i_0}x_0) \leq \Kp$
and $d_v(f^{i_1 + 1}x_0,f^n x_0) \leq \Kp$.  

Suppose that $v \neq f^jx_0$ for any $i_0 < j \leq i_1$ (by the definition of $i_0$ and $i_1$ these are the only possible indices).  Then we find that
\begin{align*}
d_v(f^{i_0}x_0,f^{i_1+1}x_0) \leq \sum_{j = i_0}^{i_1} d_v(f^jx_0,f^{j+1}x_0) = \sum_{j = i_0}^{i_1} d_{f^{-j}v}d(x_0,f x_0) \leq NM_1 \leq NM.
\end{align*}   
Else, we have that $v = f^{j_0}x_0$ for some $i_0 < j_0 \leq i_1$.  In this case we find
\begin{align*}
d_v(f^{i_0}x_0,f^{i_1+1}x_0) & \leq \sum_{j = i_0}^{j_0-2} d_v(f^jx_0,f^{j+1}x_0) + d_{f^{j_0}x_0}(f^{j_0-1}x_0,f^{j_0+1}x_0) + \sum_{j = j_0+1}^{i_1} d_v(f^jx_0,f^{j+1}x_0) \\
& (j_0 - 1 - i_0)M_1 + M_2 +  (i_1 - j_0)M_1 \leq (N-2)M_1 + M_2 \leq NM.
\end{align*} 
Therefore
\[ d_v(x_0,f^nx_0) \leq d_v(x_0,f^{i_0}x_0) + d_v(f^{i_0}x_0,f^{i_1+1}x_0) + d_v(f^{i_1+1}x_0,f^n x_0) \leq NM + 2\Kp = B_f. \qedhere \]
\end{proof}

\begin{proof}[Proof of Theorem~\ref{thm:wpd action on quotient}]
Let $\P$ be a projection complex and let $G$ be a group with a non-elementary WPD action on $\P$.

Let $f_1$ and $f_2$ be independent WPD elements in $G$.  Fix some point $x_0$ in $\P$ and let $B_{f_1}$ and $B_{f_2}$ be the constants from Lemma~\ref{lem:bounded projections}.  Let $B_0 = \max\{ d_v(f_1x_0,f_2x_0) \mid v \notin \{ f_1x_0,f_2x_0 \} \}$.  Set $B = B_0 + B_{f_1} + B_{f_2}$.  Let $v$ be a vertex of $\P$ and suppose that $d_v(f_1^{n_1}x_0,f_2^{n_2}x_0)$ is defined for some integers $n_1$ and $n_2$.  If $v \neq x_0$, then
\[ d_v(f_1^{n_1}x_0,f_2^{n_2}x_0) \leq d_v(f_1^{n_1}x_0,x_0) + d_v(x_0,f_2^{n_2}x_0) \leq B_{f_1} + B_{f_2} \leq B. \]
Else, if $v = x_0$, then
\begin{align*} 
d_v(f_1^{n_1}x_0,f_2^{n_2}x_0) \leq  d_v(f_1^{n_1}x_0,f_1x_0) + d_v(f_1x_0,f_2x_0) +  d_v(f_2x_0,f_2^{n_2}x_0) 
\leq B_{f_1} + B_0 + B_{f_2} = B. 
\end{align*}

Let $\Lwpd(\P,G) = \max\{\Llift(B),\Lpro(B)\}$.  Suppose that $\{R_v\}$ is an equivariant $L$--spinning family of subgroups of $G$ where $L \geq \Lwpd(\P,G)$.  Let $H = \grp{R_v}$.

By Lemma~\ref{lem:still wpd}, the images $\bar{f}_1$ and $\bar{f}_2$ are WPD elements of $G/H$ acting on $\P/H$.  Additionally, by Lemma~\ref{lem:project geo}, we have that $d_{\P/H}(\bar{f}_1^{n_1}\bar{x}_0,\bar{f}_2^{n_2}\bar{x}_0) = d_{\P}(f_1^{n_1}x_0,f_2^{n_2}x_0)$ for integers $n_2$ and $n_2$.  As $f_1$ and $f_2$ are independent, this shows that $\bar{f}_1$ and $\bar{f}_2$ are independent as well.  
\end{proof}

\section{Examples}\label{sec:examples}

In this final section we present two examples when $G$ is $\Mod(S)$, the mapping class group of an orientable surface $S$.  In the first example, the subgroup $H$ is the normal closure of a pseudo-Anosov mapping class; in the second example the subgroup $H$ is the normal closure of a partial pseudo-Anosov defined on an orbit-overlapping subsurface.  The first example in fact applies more generally, whenever $G$ is a group acting on a $\delta$--hyperbolic metric space and $g$ is a WPD element for this action.  The relevant background material and definitions relating to the mapping class group that appear in this section can be found in our previous paper with Margalit~\cite{un:CMM}.

Before we give the examples, we first recall the criteria of Bestvina--Bromberg--Fujiwara for showing that an element $g$ of $G$ acts on a projection complex $\P$ as a WPD element.

\ip{WPD criterion.}  Suppose that $\P$ is a projection complex.  Bestvina--Bromberg--Fujiwara proved the existence of a constant $\Kwpd$ which can be used to ensure that an element acting on $\P$ is a WPD element~\cite[Proposition~3.27]{ar:BBF15}.  The set-up is as follows.  Assume that $G$ is a group that acts on $\P$ and that $g$ is an element of $G$ that satisfies the following two conditions.
\begin{enumerate}
\item There is a vertex $v$ in $\P$ and $n > 0$ such that $d_v(g^{-n}v,g^n v) > \Kwpd$; and
\item there is an $m > 0$ such that the subgroup of $G$ that fixes $v, gv, \ldots, g^m v$ is finite.
\end{enumerate}  
Then $g$ is a WPD element of $G$.

\subsection{First example}\label{sec:first example} Let $S$ be an orientable surface where $\chi(S) < 0$ and let $f$ be a pseudo-Anosov mapping class of $S$.  There is a projection complex $\P$ built using $f$ and its action on the curve complex $\C(S)$.  We briefly recall this construction here; full details can be found in our previous paper~\cite[Section~3.2]{un:CMM}.  

Fix a point $x$ in $\C(S)$ and consider the \emph{quasi-axis bundle} $\beta = \EC(f) \cdot x$, where $\EC(f)$ is the \emph{elementary closure of $f$}.  In this context, $\EC(f)$ is the stabilizer of the set of transverse measured foliations associated to $f$ considered in the space of projectivized measured foliations on $S$.  

The vertex set of $\P$ consists of the $\Mod(S)$--translates of $\beta$.  Next we define the distance functions.  Given three vertices $\alpha_1$, $\alpha_2$ and $\beta$ of $\P$, we define $d_\beta(\alpha_1,\alpha_2)$ to be diameter of the union of the projections of $\alpha_1$ and $\alpha_2$ to $\beta$.

Let $g$ be a mapping class of $S$ that does not lie in $\EC(f)$.  We claim that $gf^n$ is a WPD element for the action on $\P$ for $n$ sufficiently large.  As $g$ does not lie in $\EC(f)$, we have that $\beta \notin \{g\beta,g^{-1}\beta\}$.  Next, we have that \[d_\beta((gf^n)^{-1}\beta,gf^n\beta) = d_\beta(g^{-1}\beta,f^ng\beta) \geq d_\beta(g\beta,f^ng\beta) - d_\beta(g\beta,g^{-1}\beta)\] and thus $d_\beta((gf^n)^{-1}\beta,gf^n\beta)$ is bounded below by $An - B$ for some constants $A,B > 0$.  In particular, $d_\beta((gf^n)^{-1}\beta,gf^n\beta) > \Kwpd$ for sufficiently large $n$.  Further, as $g$ does not stabilize the measured foliations associated to $f$, the stabilizer of $\beta$ and $g\beta$ is finite; see~\cite[Proposition~4.7]{un:BBF-v1}.  By the Bestvina--Bromberg--Fujiwara WPD criterion we have that $f_1 = gf^n$ is a WPD element for some fixed sufficiently large enough $n$.  

Given an element $h$ that does not lie in $\EC(f_1)$, the element $f_2 = hf_1h^{-1}$ is a WPD element and $f_1$ and $f_2$ are independent.  Thus the action of $\Mod(S)$ on $\P$ is a non-elementary WPD action.

As explained in the proof of Theorem~1.7 in our previous work, for each $L \geq 0$, there is an $p > 0$ such that the collection of subgroups $R_{h\beta} = \grp{hf^ph^{-1}}$ is an equivariant $L$--spinning family of subgroups.  Therefore, by Theorem~\ref{thm:wpd action on quotient}, the elements of $\bar f_1$ and $\bar f_2$ in $\Mod(S)/\nc{f^p}$ are independent WPD elements for its action on $\P/\nc{f^p}$ for a certain large enough $p$.

As mentioned at the beginning of this section, this above discussion works in the larger context of a group $G$ acting on a $\delta$--hyperbolic space using a WPD element $f$ of $G$.

\subsection{Second example}\label{sec:second example} Again, let $S$ be an orientable surface where $\chi(S) < 0$.  Let $X$ be a connected subsurface of $S$ so that for all $h \in \Mod(S)$, either $X = hX$ or $X$ and $hX$ have nontrivial intersection (what is called an \emph{orbit-overlapping} subsurface in our previous work~\cite{un:CMM}).     There is a projection complex $\P$ built using $X$ and the curve complex $\C(S)$.  We briefly recall this construction here; full details can be found in our previous paper~\cite[Section~3.3]{un:CMM}.  

The vertices of $\P$ are the $\Mod(S)$--translates of $X$.  Given three vertices $Y_1$, $Y_2$ and $X$ of $\P$, the distance $d_X(Y_1,Y_2)$ is the diameter in $\C(X)$ of the Masur--Minsky subsurface projections of $Y_1$ and $Y_2$ to $X$~\cite{ar:MM00}.      

There is a well-defined map $\Mod(X) \to \Mod(S)$; fix an element $f$ in $\Mod(S)$ that is the image of a pseudo-Anosov element on $X$.  Let $g$ be a mapping class of $S$ such that $\partial X$ and $g\partial X$ fill $S$.  We claim that $gf^n$ is a WPD element for the action on $\P$ for sufficiently large $n$.  The proof is similar to the first example and left to the reader.

Hence, as above, there are elements $f_1$ and $f_2$ in $\Mod(S)$ that are independent WPD elements for the action on $\P$.  By taking certain $p$ sufficiently large, we can ensure that the equivariant family of subgroups $R_{hX} = \nc{hf^ph^{-1}}_{\Stab(hX)}$ is $L$--spinning for arbitrary $L$.  Hence we can ensure that the images $\bar f_1$ and $\bar f_2$ are independent WPD elements for the action of $\Mod(S)/\nc{f^p}$ on $\P/\nc{f^p}$.
  
Similar arguments apply to the other subgroups of $\Mod(S)$ constructed in our previous work.


\bibliography{hypquotient}
\bibliographystyle{acm}

\end{document}